\documentclass[sn-mathphys,Numbered]{sn-jnl}

\usepackage[french, english]{babel}
\usepackage{graphicx}%
\usepackage{multirow}%
\usepackage{amsmath,amssymb,amsfonts,amscd, amsbsy}%
\usepackage{amsthm}%
\usepackage{mathrsfs}%
\usepackage[title]{appendix}%
\usepackage{xcolor}%
\usepackage{textcomp}%
\usepackage{manyfoot}%
\usepackage{booktabs}%
\usepackage{algorithm}%
\usepackage{algorithmicx}%
\usepackage{algpseudocode}%
\usepackage{listings}%
\usepackage{enumerate}

\usepackage{ulem}

\theoremstyle{thmstyleone}%
\theoremstyle{thmstyletwo}%

\theoremstyle{thmstylethree}%

\newcommand*{\BF}[1]{\ifmmode\bm{#1}\else\textbf{#1}\fi}

\newcommand\R{{\mathbb R}}

\newcommand{\so}[1]{{\color{black} #1}}

\newtheorem{Theorem}{Theorem}
\newtheorem{Proposition}{Proposition}

\newtheorem{Lemma}{Lemma}

\newtheorem{Remark}{Remark}

\DeclareFontFamily{U}{BOONDOX-calo}{\skewchar\font=45 }
\DeclareFontShape{U}{BOONDOX-calo}{m}{n}{
  <-> s*[1.05] BOONDOX-r-calo}{}
\DeclareFontShape{U}{BOONDOX-calo}{b}{n}{
  <-> s*[1.05] BOONDOX-b-calo}{}
\DeclareMathAlphabet{\mathcalboondox}{U}{BOONDOX-calo}{m}{n}
\SetMathAlphabet{\mathcalboondox}{bold}{U}{BOONDOX-calo}{b}{n}
\DeclareMathAlphabet{\mathbcalboondox}{U}{BOONDOX-calo}{b}{n}
\newcommand{\mcb}[1]{{\mathcalboondox #1}}

\begin{document}

\title[]{Space-time fluctuations in a quasi-static limit}


\author*[1]{\fnm{C\'edric} \sur{Bernardin}}\email{cbernard@unice.fr}

\author[2]{\fnm{Patricia} \sur{Gon\c{c}alves}}\email{pgoncalves@tecnico.ulisboa.pt}
\equalcont{These authors contributed equally to this work.}

\author[3,4,5]{\fnm{Stefano} \sur{Olla}}\email{olla@ceremade.dauphine.fr}
\equalcont{These authors contributed equally to this work.}

\affil*[1]{\orgdiv{Faculty of Mathematics}, \orgname{National Research University Higher School of Economics},  \orgaddress{\street{6 Usacheva}, \city{Moscow}, \postcode{119048 }, \country{Russia}}}

\affil[2]{\orgdiv{Center for Mathematical Analysis,  Geometry and Dynamical Systems}, \orgname{Instituto Superior T\'ecnico, Universidade de Lisboa}, \orgaddress{\street{Av. Rovisco Pais}, \city{Lisboa}, \postcode{1049-001}, \country{Portugal}}}

\affil[3]{\orgdiv{CEREMADE}, \orgname{Universit\'e Paris Dauphine - PSL Research University}, \orgaddress{\street{Place du Mar\'echal De Lattre De Tassigny}, \city{Paris}, \postcode{75775 Cedex 16}, \country{France}}}

\affil[4]{\orgdiv{Institut Universitaire de France}, \orgaddress{\street{1 rue Descartes }, \city{Paris}, \postcode{75231 Cedex 05}, \country{France}}}

\affil[5]{\orgdiv{Gran Sasso Science Institute}, \orgaddress{\street{ Viale F. Crispi, 7}, \city{L'Aquila}, \postcode{67100}, \country{Italy}}}


\abstract{ We consider the macroscopic limit for the space-time density fluctuations in the open
  symmetric simple exclusion in the quasi-static scaling limit.
  We prove that the distribution of these fluctuations converge
  to a gaussian space-time field that is delta correlated in time
  but with long-range correlations in space.}

\keywords{Interacting particle systems, Transport processes, Long-range correlations}


\pacs[MSC Classification]{82C22,82C70,60K35}

\maketitle

\section{Introduction}
\label{intro}

Consider a stochastic dynamics  of interacting particles where the only conserved quantity
is given by the number of particles, in equilibrium with a reversible stationary (equilibrium) Gibbs
measure $\mu_\rho$, $\rho$ indicates here the average density of particles.
The fluctuation field of the density of particles at a given time $t$ is defined by averaging in space
on a macroscopic scale $N$.  
After a \emph{diffusive} rescaling of space and time the fluctuation field
in space of the density of particles converges
to a distribution ${\mcb Y}_t(u)$ expected to evolve in time
following the infinite dimensional Langevin equation (\cite{Spohn}, section II.2.9)
\begin{equation}
  \label{eq:20}
  \partial_t {\mcb Y}_t(u) = \mathcal D_\rho \partial_u^2  {\mcb Y}_t(u) -
  \sqrt{2\chi_\rho \mathcal D_\rho} \partial_u w(t,u),
\end{equation}
where $\mathcal D_\rho$ is the diffusivity at density $\rho$, $\chi_\rho$ is the static compressibility
and $w(t,u)$ is a standard space-time white noise. The equation \eqref{eq:20} shall be
considered in a weak sense and we wrote it in one dimension for simplicity.
The only stationary distribution for \eqref{eq:20} is given by the centered Gaussian
distribution on with covariance
\begin{equation}
  \label{eq:21}
  \mathbb E\left({\mcb Y}_t(u) {\mcb Y}_t(v)\right) =  \chi_\rho \delta(u-v),
\end{equation}
that correspond to the static fluctuation field of the stationary Gibbs measure,
that typically have exponential decay of space correlations on the microscopic scale.
The time correlations of the solution of  \eqref{eq:20} have exponential decay.

In this article we consider the \emph{space-time} fluctuations of the density of a open system 
on the space-time box: $ \{1,\ldots, N-1\} \times [0, N^{2+\alpha}T]$. When $\alpha>0$ as $N\to\infty$
this correspond to
a \emph{quasi-static} scaling limit, where the time scale is larger than the typical scale where the
dynamics relax to equilibrium (i.e. the diffusive time scale with $\alpha = 0$ \cite{DO-qs}).

We expect that the limit space-time field is a centered gaussian field $\mathbb Y(t,u)$ with covariance
\begin{equation}
  \label{eq:22}
   \mathbb E\left({\mathbb Y}(t,u) {\mathbb Y}(s,v) \right) =  \chi_\rho \delta(t-s) (-\partial_u^2)^{-1}(u,v),
 \end{equation}
 where $(-\partial_u^2)^{-1}(u,v)$ is the Green's function of the laplacian on $[0,1]$
 with \so{Dirichlet} boundary conditions.
While the delta-correlation in time is a natural consequence of the quasi-static time scale,
at first sight it may look surprising the long range correlation in space appearing in \eqref{eq:22}.

In the present article we prove such macroscopic behaviour for the open
symmetric simple exclusion on the discrete interval $\{1, \dots, N-1\}$
with boundary reservoirs, but we go beyond the equilibrium case (reservoirs at the same density $\rho$)
and we consider reservoirs at different densities ($\rho_-(t)$ on the left and $\rho_+(t)$ on the right)
that change in time on the macroscopic quasi-static scale, in the same situation as considered in
\cite{DO-qs}. When $\rho_\pm$ do not change in time, a stationary state exists that has a density profile
$\rho(u) = (\rho_+ - \rho_-)u + \rho_-$ and has long range
correlations.
The static fluctuation field in this stationary state  in the limit $N\to\infty$
has covariance \cite{Spohn_2,lmo}
\begin{equation}
  \label{eq:23}
  \mathbb E\left({\mcb Y}(u) {\mcb Y}(v)\right) =
  \chi_{\rho(u)} \delta(u-v) - \left(\rho_{+}- \rho_{-} \right)^2  (-\partial_u^2)^{-1}(u,v).
\end{equation}
where $\chi_\rho = \rho (1-\rho)$.
A similar limit is found for the density fluctation at macroscopic time $t$ in \cite{DO-qs} when $\rho_\pm(t)$
are time dependent. Then we prove in this situation that the space-time fluctuation
field converges to a centered gaussian field with covariance defined by \eqref{eq:VF} \eqref{eq:VFe}.

Even though long-range space correlations related to the Green function of the Dirichlet Laplacian
appears in \eqref{eq:22} (and its generalization \eqref{eq:VF}) as well as in \eqref{eq:23},
there is not a direct connection between the two results. Notice that in the covariance of the space-time
fluctuations \eqref{eq:22} the long range correlations appear also in equilibrium (i.e. $\rho_-= \rho_+$),
while in the space-fluctuations in the stationary state they appears only if $\rho_-\neq \rho_+$.
\so{Furthermore in \eqref{eq:22} there is no term with a delta correlation in space and the long range
  terms appear with a positive sign (while they are negative in \eqref{eq:23}).}

In Section \ref{sec:open-dynamics-with} we present the model  under investigation, we recall from \cite{DO-qs} the hydrodynamic limit but also estimates on correlation functions (which are crucial to achieve our results) and  we state our main theorem.
We carry out the complete proof  of our main theorem for the open SSEP in Section 3, while in Section 4 we sketch the extension to other \so{symmetric} stochastic dynamics like the zero-range model.

\so{We believe the results of the present article can be easily extended to multidimensional cases,
  for example on the open square of side length $N$. In the equilibrium case, with reservoirs at the boundaries with density $\rho$ we expect that the macroscopic space-time fluctuations are distributed
  like the gaussian field with covariance
\begin{equation}
  \label{eq:22d}
   \mathbb E\left({\mathbb Y}(t,u) {\mathbb Y}(s,v) \right) =  \chi_\rho \delta(t-s) (-\Delta)^{-1}(u,v),
 \end{equation}
 where $(-\delta^2)^{-1}(u,v)$ is the Green's function of the laplacian on $[0,1]^d$
 with {Dirichlet} boundary conditions.
 For time changing boundaries densities and in non-equilibrium, under reasonable conditions we expect the corresponding generalization of  \eqref{eq:VF} \eqref{eq:VFe} should hold.
}

\so{
  Another interesting generalization would be the investigazion of the space-time fluctuation for the
  asymmetric dynamics, like the ASEP where there exists results for the quasi-static limit \cite{DMLO}.
  }

\section{SSEP with boundary reservoirs}
\label{sec:open-dynamics-with}
\subsection{The model}
Let $\alpha>0$. We consider the symmetric simple exclusion process $\{\eta_{t} \; ; \;  t\ge 0\}$ on the interval $\{1, \dots, N-1\}$, with boundary creation/destruction of particles with rates $\rho_{-}(t)$ at $x=1$ and $\rho_{+}(t)$ at $x=N-1$. The generator of the time inhomogeneous pure jumps Markov process $\{\eta_{tN^{2+\alpha}} \; ; \;  t\ge 0\}$ with state space $\Omega^N=\{0,1\}^{\{1, \dots, N-1\}}$ is given by 
\begin{equation*}
{\mcb L}_{N,t} =N^{2+\alpha} \left({\mcb L}_N^0 +\mcb L_{N,t}^{\pm} \right)
\end{equation*}
where the generators above are acting on functions $f:\Omega^N \to \mathbb R$ according to
\begin{equation*}
(\mcb L_N^0 f )(\eta) =\sum_{x=1}^{N-2}\  \left[ f(\eta^{x,x+1}) -f(\eta) \right]
\end{equation*}
 and
 \begin{equation*}
   (\mcb L_{N,t}^{-} f )(\eta) =\left[1-\rho_- (t)\right] \eta (1)\  \left[ f(\eta^{1}) -f(\eta) \right]
   +\rho_- (t) (1-\eta (1))\ \left[ f(\eta^{1}) -f(\eta) \right] 
\end{equation*}
and
\begin{equation*}
(\mcb L_{N,t}^{+} f )(\eta) =\left[1-\rho_+ (t)\right] \eta (N-1)\  \left[ f(\eta^{N-1}) -f(\eta) \right]
  +\rho_+ (t) (1-\eta (N-1))\ \left[ f(\eta^{N-1}) -f(\eta) \right].
\end{equation*}
Above, $\eta^{x,y}$ is the configuration obtained from $\eta$ by exchanging the occupation variables at sites $x$ and $y$, while $\eta^x$ is the configuration obtained from $\eta$ by flipping the occupation variable at site $x$, that is:
\begin{equation*}
\eta^{x,y}(z)=\begin{cases}
\eta(z)& \text{ if } z\notin \{x,y\},\\
\eta(y) & \text{ if } z=x,\\ 
\eta(x) & \text{ if } z=y,
\end{cases}
\end{equation*} and 
\begin{equation*}
\eta^{x}(y)=\begin{cases}\eta(y) & \text{ if }y\neq x, \\ 1-\eta(x) & \text{ if } y=x. \end{cases}
\end{equation*}

\subsection{Hydrodynamic limit}
It is proved in \cite{DO-qs} the following Law of Large Numbers:
if {$\rho_0:[0,1]\to[0,1]$ is a measurable profile and}
at initial time the (deterministic or random) initial probability measure $\mu_N$ on $\Omega^N$ is such that 
\begin{equation}
  \label{eq:15-1}
  \begin{split}
  &\frac 1N \sum_{x=1}^{N-1} G\left(\cfrac xN\right) \eta_{0}(x) \
  \mathop{\longrightarrow}^{\rm{prob.}}_{N\to\infty} \ \int_0^1 G(u) \rho_0(u) \; du,
\end{split}
\end{equation}
for any continuous function $G:[0,1] \to \mathbb R$, then at any time $t\ge 0$ it holds
\begin{equation}
  \label{eq:15}
  \begin{split}
  &\frac 1N \sum_{x=1}^{N-1} G\left(\cfrac xN\right) \eta_{tN^{2+\alpha}}(x) \
  \mathop{\longrightarrow}^{\rm{prob.}}_{N\to\infty} \ \int_0^1 G(u) {\rho}_t(u)\; du,\\
  &\text{with\quad} \rho_t(u) = \left(\rho_{+}(t) - \rho_{-}(t)\right) u + \rho_{-}(t).
\end{split}
\end{equation}
Let $\mu_N$ be a probability measure on $\Omega^N$ and,  for $t \ge 0$,
we denote
\begin{equation}
\label{eq:barrhon}
\rho^N_t\left(x\right) := \mathbb E_{\mu_N} \left[ \eta_{tN^{2+\alpha}} (x)  \right], \quad x\in\{1,\ldots, N-1\},
\end{equation}
with $\rho^N_t\left(0\right) := \rho_-(t)$, and $\rho^N_t\left(N\right) := \rho_+(t)$. Define for any $u,v \in [0,1]$ and $s,t \ge 0$
\begin{equation}
\label{eq:barphin}
{\varphi}^N (u,v; s,t) :=
\mathbb E_{\mu_N}\left[\left(\eta_{sN^{2+\alpha}}([Nu]) - \rho^N_s ([Nu]) \right)
      \left(\eta_{tN^{2+\alpha}}([Nv]) - \rho^N_t ([Nv]) \right)\right].
\end{equation}
We assume that
\begin{equation}
  \label{eq:assumption}
  \begin{split}
   \sup_{N\in\mathbb N} \sup_{u,v \in [0,1], u\neq v} N  | {\varphi}^N (u,v; 0,0) | \lesssim 1
  \end{split}
\end{equation}
where here and in the following, $\lesssim$ denotes an inequality that is correct up to a multiplicative constant independent of $N$.

It is proved in Theorem 2.2 of  \cite{DO-qs} that, for $u\neq v$ and $u,v\not\in\{0,1\}$,
\begin{equation}
  \label{eq:16}
  \begin{split}
    &\lim_{N\to\infty} N  {\varphi}^N (u,v; t,t)  = - \left(\rho_{+}(t) - \rho_{-}(t)\right) u(1-v)
  \end{split}
\end{equation}
and for $t\neq s$
\begin{equation}
  \label{eq:19}
    \lim_{N\to\infty} N {\varphi}^N (u,v; s,t) = 0.
\end{equation}

We also note that from Lemma  4.3 in \cite{lmo} and from \eqref{eq:assumption} we get that
\begin{equation}
\label{eq:boundcorr}
\sup_{N\ge 1} \max_{x,y\in\{1,\cdots,N-1\},  x\neq y} \,  N \vert  {\varphi}^N \big( \tfrac xN , \tfrac yN ; t,t \big) \vert \lesssim 1.
\end{equation}
Above we used the proof of \cite{lmo} that can be adapted to our case, since our correlation  function   ${\varphi}^N (u,v; t,t)$ (see \eqref{eq:barphin} with $s=t$) also satisfies an equation similar to (4.5) in \cite{lmo}. By following exactly the same steps as in the proof of Lemma 4.3 of \cite{lmo} we conclude \eqref{eq:boundcorr}. We leave these details to the interested  reader. 
\subsection{Main result}
For a fixed $T>0$, we denote the open set $\Omega_T= (0,T)\times (0,1) \subset \R^2$.
  We define the (centered) space-time distribution valued random field
  $\mathbb Y^N$ acting on 
  $f\in \mathcal C^{\infty}\left(\Omega_T, \mathbb C\right)$ as
\begin{equation}\label{eq:Yqs}
  \mathbb Y^N (f) 
  =  \int_0^T \cfrac{1}{\sqrt {N^{1-\alpha}}} \sum_{x=1}^{N-1} {f}
  (t, \tfrac x N) \left(\eta_{t N^{2+\alpha}} (x) -
    {\rho^N_t(x)}\right) dt.
\end{equation}

Let $\mathbb Y$ be the space-time distribution valued {centered}
Gaussian field such that for any  function $f$
the variance $\mathbb V (f)$ of $\mathbb Y (f)$ is given by
\begin{equation}
  \label{eq:VF}
     \mathbb V (f) = 2 \int_0^T ds \int_0^1 du\  \rho_s(u) (1- \rho_s(u))
     \left\vert \partial_u F (s,u) \right\vert^2
\end{equation}
where $F$ is the solution of
\begin{equation}\label{eq:VFe}
\begin{cases}
&\left(\partial_u^2 F\right) (t,u) =f (t,u), \quad u \in [0,1],\\
&F(t,0)=F(t,1)=0,
\end{cases}
\end{equation}
{for all $t\in[0,T]$.}

Our main result is the following.
\bigskip

\begin{Theorem}
\label{th:mainthm1}
The sequence  $\{\mathbb Y^N \; ; \; N\ge 1\}$ converges in law,  as $N\to+\infty$, to $\mathbb Y$.
\end{Theorem}

\bigskip

\begin{Remark}
 We have defined $\mathbb Y^N$ as an element of the dual of $\mathcal C^{\infty}\left(\Omega_T, \mathbb C\right)$. In the theory of distributions is usually denoted $\mathcal E \left(\Omega_T, \mathbb C\right) = \mathcal C^{\infty}\left(\Omega_T, \mathbb C\right)$ and $\mathcal D \left(\Omega_T, \mathbb C\right) = \mathcal C_0^{\infty}\left(\Omega_T, \mathbb C \right)$. The restriction from $\mathcal E' \left(\Omega_T, \mathbb C \right)$ to $\mathcal D' \left(\Omega_T, \mathbb C \right)$ is injective (cf. \cite{DK}, Chapter 8). In Subsection \ref{tight} we prove that $\{ \mathbb Y^N \; ; \; N\ge 1\}$ is tight in a Sobolev space $\mcb H_{-m,-k} \subset \mathcal E^\prime \left(\Omega_T, \mathbb C \right)$.
\end{Remark}

\section{Proof of Theorem \ref{th:mainthm1}}

The proof of Theorem \ref{th:mainthm1} is a consequence of two ingredients: convergence of finite dimensional distributions and tightness. 
Therefore,  in Proposition \ref{Prop:con-law} we establish
   the convergence in law of the sequence $\{\mathbb Y^N (f)\; ; \; N\ge 1\}$ to $\mathbb Y (f)$
   for any test function $f$ and  in Lemma \ref{lem:tightness}
   we prove  that the sequence $\{\mathbb Y^N \; ; \; N\ge 1\}$
   is tight in a suitable Sobolev space of distributions.

\subsection{Convergence of finite-dimensional distributions}

\begin{Proposition}
\label{Prop:con-law}
For any $\theta \in \mathbb R$ and for  any $f\in C^\infty(\Omega_T, \mathbb C)$ it holds that the sequence $\{\mathbb Y^N (f) \; ; \; N \ge 1\}$ converges in law to the centred complex Gaussian random variable with variance $\mathbb V (f)$ where $\mathbb V (f)$ is defined in \eqref{eq:VF}.
  \end{Proposition}
\begin{proof}
We claim first that it is sufficient to prove  that
\begin{equation}
\label{eq:14}
\lim_{N\to\infty} \mathbb E\left( \exp\left\{i\theta {\mathbb Y}^N (f)\right\}\right) = \exp\left\{ -\tfrac{\theta^2}{2} \mathbb V (f)   \right\}
\end{equation}
for real-valued functions.
\medskip
Indeed, if this is the case, then let $f \in C^\infty (\Omega_T, \mathbb C)$ be decomposed as $f=g+ i h$ where $g,h \in C^\infty (\Omega_T, \mathbb R)$. Therefore, for any real numbers $\lambda, \mu$, \eqref{eq:14} is valid for $\lambda g + \mu h$, i.e. the sequence of real random variables 
\begin{equation*}
\left\{ \mathbb Y^N (\lambda g + \mu h) = \lambda \mathbb Y^N (g) + \mu \mathbb Y^N (h) \; ; \; N\ge 1\right\}
\end{equation*}
converges to a Gaussian random variable with variance $\mathbb V ( \lambda g + \mu h)$. Let us define
\begin{equation}
\mathbb V (g,h)= 2 \int_0^T ds \int_0^1 du\  \rho_s(u) (1- \rho_s(u))
      \partial_u G (s,u) \ \partial_u H (s,u)
\end{equation}
where $G$ and $H$ are the solutions, respectively,  of
\begin{equation}
\begin{cases}
&\left(\partial_u^2 G\right) (t,u) =g (t,u), \quad u \in [0,1],\\
&G(t,0)=G(t,1)=0,
\end{cases}
\end{equation}
for all $t\in[0,T]$,
and 
\begin{equation}
\begin{cases}
&\left(\partial_u^2 H\right) (t,u) =h (t,u), \quad u \in [0,1],\\
&H(t,0)=H(t,1)=0,
\end{cases}
\end{equation}
for all $t\in[0,T]$. It is not difficult to check from the definition of $\mathbb V$ given in \eqref{eq:VF} that $\mathbb V (\lambda g + \mu h)= \lambda^2 \mathbb V (g) + \mu^2 \mathbb V (h) +2\lambda\mu \mathbb V (g,h)$. This implies that the sequence of random vectors 
\begin{equation*}
\left\{ ( \mathbb Y^N (g),  \mathbb Y^N (h) )\; ; \; N\ge 1\right\}
\end{equation*}
converges in law to a centred Gaussian random vector with covariance matrix
 \begin{equation*}
\begin{pmatrix}
\mathbb V (g) & \mathbb V (g, h)\\
\mathbb V (g,h)& \mathbb V(h)
\end{pmatrix} \ .
\end{equation*}
Hence the sequence $\{ \mathbb Y^N (f)=\mathbb Y^N (g) +i \mathbb Y^N (h)   \; ;\; N\ge 1\}$ converges in law to a complex centered Gaussian variable with variance  (observe that the solution of \eqref{eq:VFe} is $F=G+i H$)
\begin{equation*}
\mathbb V (g) +\mathbb V (h) =\mathbb V (f)
\end{equation*}
and this proves the claim.
\medskip
Now, it remains to prove \eqref{eq:14} for real-valued functions.

Note that from \eqref{eq:Yqs} we have 
\begin{equation*}
  \mathbb Y^N (f) = N^{\alpha/2} \int_0^T dt\,\, {\mcb Y}_t^N (f)
\end{equation*}
where the \emph{space} fluctuation field $\{ {\mcb {Y}}_t^N\; ; \;  t\ge 0\}$
is defined on a space-time function $f: [0,T]\times [0,1] \to \mathbb R$ as
\begin{equation*}
  {\mcb{ Y}}_t^N (f) := \cfrac{1}{\sqrt N} \sum_{x=1}^{N-1}
  {f}( t, \tfrac xN)\left(\eta_{t N^{2+\alpha}} (x) -
    {\rho^N_t(x)}\right).
\end{equation*}
With the convention that $\eta_t(0)=\rho_-(t) $ and $\eta_t(N)=\rho_+(t)$ for all $t\in[0,T]$, we observe that
\begin{equation*}
\begin{split}
&\mcb L_{N,t} (\eta_t (x)) = N^\alpha (\Delta_N \eta_t )(x), \quad x\in\{1, \ldots, N-1\},
\end{split}
\end{equation*}
where $(\Delta_N \eta_t )(x)=N^2\left(\eta_t(x+1)+\eta_t(x-1)-2\eta_t(x)\right)$.
In particular,
\begin{equation*}
\begin{split}
\partial_t \left[{\rho}^N_t (x)\right] &= N^\alpha ( \Delta_N {\rho}^N_t )(x), \quad  x\in\{1, \ldots, N-1\}.
\end{split}
\end{equation*}

Let $G: (t,u) \in [0,T]\times [0,1] \to \mathbb R$ be a differentiable function in time with Dirichlet boundary conditions, i.e., for any $t\in [0,T]$
\begin{equation*}
G(t,0)= G(t, 1)=0.
\end{equation*}

Since  $\Delta_N {\rho}^N_s (x) = 0$,
by Dynkin's formula it follows that the process $\{  {\mcb M}_t^N( G) \; ; \; t \ge 0\}$ defined by
\begin{equation}
\begin{split}
  {\mcb M}_t^N( G)&=\mcb{Y}^N_t( {{G}})-{\mcb{Y}}^N_0( {{ G}}) \\
  &-\int_0^t \cfrac{1}{\sqrt N} \sum_{x=0}^N
  {\partial_s G}( s, \tfrac xN) (\eta_{s N^{2+\alpha}} (x) -{ \rho}^N_s (x)) ds\\
  &- N^{\alpha -1/2}  \int_0^t \sum_{x=1}^{N-1} G( s, \tfrac xN) \left[ \Delta_N \left(\eta_{sN^{2+\alpha}} (x)
      - {\rho}^N_s (x) \right)\right] ds 
\end{split}
\end{equation}
is a martingale w.r.t. the natural filtration of $\{\eta_{tN^{2+\alpha}} \; ; \; t\ge 0\}$. By performing a  summation by parts  and noting that $G$ satisfies  the Dirichlet boundary conditions, we get 
\begin{equation*}
\begin{split}
  {\mcb M}_t^N( G)&=\mcb{Y}^N_t( {{G}})-{\mcb{Y}}^N_0( {{ G}}) -\int_0^t {\mcb Y}_s^N
  (\partial_s G + N^\alpha \widetilde\Delta_N G) ds,
\end{split}
\end{equation*}
where $\widetilde\Delta_N$ is defined by
\begin{equation*}
\begin{split}
  &(\widetilde\Delta_N G)(u) =N^2 \left[ G(u+\tfrac 1N) +G(u-\tfrac 1N) -2G(u)\right], \quad
  u\in\{\tfrac 1N, \ldots, \tfrac {N-1}N\}. 
\end{split}
\end{equation*}
It follows that
\begin{equation}
\label{eq:007}
\begin{split}
{\mathbb{ Y}}^N \left(\widetilde\Delta_N { G}+N^{-\alpha}{\partial_t { G}} \right)
&= \frac{1}{N^{\alpha/2}} \left\{ {\mcb{Y}}^N_T({G})-{\mcb{Y}}^N_0({G})\right\}
-\frac{1}{N^{\alpha/2}} {\mcb M}_T^N(G).
\end{split}
\end{equation}

\bigskip
Let now $f\in \mathcal C^\infty(\Omega_T,\mathbb R)$ and consider, for any $N\ge 1$, the function $F^N(t,\tfrac xN), x\in \{0, \ldots, N\}$, satisfying the finite-dimensional linear differential equation
\begin{equation}
\label{eq:Dirichlet problem}
\begin{cases}
&N^{-\alpha} \partial_t F^N(t,\tfrac xN) + \widetilde \Delta_N F^N (t,\tfrac xN)=f(t,\tfrac xN),  \quad x\in\{1,\cdots, N-1\},\quad  0<t<T,\\
&F^N(t, 0)= F^N (t, 1)=0,\\
&F^N (T,\cdot) =F(T,\cdot).
\end{cases}
\end{equation} 
The quadratic variation of the martingale $N^{-\alpha/2}{\mcb M}_T^N(F^N)$ satisfies
\begin{equation}
  \label{eq:0071}
  \begin{split}
  \mathbb V^N (f) := N \int_0^T \sum_{x=0}^{N-1} \Big(\eta_{sN^{2+\alpha}}(x)  - \eta_{sN^{2+\alpha}}(x+1) \Big)^2
  \left|F^N  \Big(s, \tfrac {x+1} N\Big) - F^N \Big(s, \tfrac x N\Big)  \right|^2\\
  \mathop{\longrightarrow}_{N\to\infty}^{\mathbb L^2}
{2 \int_0^Tds \int_0^1du\ \rho(s,u) (1- \rho(s,u))  \left\vert \partial_u F (s,u) \right\vert^2}
  = \mathbb V (f).
\end{split}
\end{equation}

To prove the previous convergence, we use \eqref{eq:H_1-app} proved below.
It is then sufficient to prove that
\begin{equation}
  \begin{split}
\int_0^T \frac{1}{N} \sum_{x=0}^{N-1} \Big(\eta_{sN^{2+\alpha}}(x) & - \eta_{sN^{2+\alpha}}(x+1) \Big)^2
  \left|(\partial_u F) (s, \tfrac {x} N)\right|^2 ds\\
  &   \mathop{\longrightarrow}_{N\to\infty}^{\mathbb L^2}  2\int_0^T \int_0^1\rho(s,u) (1- \rho(s,u))
    \vert (\partial_u F) (s,u) \vert^2    du  ds,
\end{split}
\end{equation}
but this is a simple consequence of Theorem 3.1 of \cite{DO-qs}.

\bigskip
Let us now prove \eqref{eq:14}. Recalling that $F^N (T, \cdot)= F(T, \cdot)$
we rewrite \eqref{eq:007} with $G$ replaced by $F^N$ as
\begin{equation*}
  \mathbb Y^N (f) =- N^{-\alpha/2} \left[ {\mcb{Y}}^N_0({ F^N}) - {\mcb{Y}}^N_T({ F})
    +\mcb M_T^N (F^N)\right].
\end{equation*}
We have then
\begin{equation}
\label{eq:tatu}
\begin{split}
&\mathbb E\Big[\exp\{i \theta \mathbb Y^N(f)\}\Big]\\
&=
\mathbb E\Big[\Big(\exp\Big\{-i\theta N^{-\alpha/2}
\left(\mcb Y_0^N (F^N) - {\mcb{Y}}^N_T({ F})\right)  \Big \} -1\Big) 
\exp\Big\{-i\theta N^{-\alpha/2} \mcb M_T^N (F^N)\Big\}\Big]\\
&+\mathbb E\Big[\exp\Big\{-i\theta N^{-\alpha/2} {\mcb M}_T^N (F^N)
+\frac{\theta^2}{2}\mathbb V^N (f)\Big\}\exp\Big\{-\frac{\theta^2}{2} \mathbb V^N (f) \Big\}\Big].
\end{split}
\end{equation}
The modulus of the first term on the right-hand side of the
previous display is bounded from above by
\begin{equation}
\label{eq:L2boundY}
\begin{split}
  &\Big|\mathbb E\Big[\Big(\exp\Big\{-i\theta N^{-\alpha/2}
 \left(\mcb Y_0^N (F^N) - {\mcb{Y}}^N_T({ F})\right)  \Big \} \Big\}-1\Big)
  \exp\Big\{-i\theta N^{-\alpha/2} \mcb  M_T^N (F^N)\Big\}\Big]\Big|\\
  &\leq  \mathbb E\Big[\Big|\exp\Big\{-i\theta N^{-\alpha/2}
  \left(\mcb Y_0^N (F^N) - {\mcb{Y}}^N_T({ F})\right)  \Big \} \Big\}-1\Big|\Big]\\
  &\lesssim { |\theta|} N^{-\alpha/2} \left(\mathbb E\Big[\Big| \mcb Y_0^N (F^N)\Big |\Big] +
   \mathbb E\Big[\Big| \mcb Y_T^N (F)\Big |\Big]\right)\\
&\leq  |\theta| N^{-\alpha/2} \sqrt{\mathbb E\Big[\Big|\mcb Y_0^N (F^N)\Big |^2\Big]}
+  |\theta| N^{-\alpha/2} \sqrt{\mathbb E\Big[\Big|\mcb Y_T^N (F)\Big |^2\Big]} \\
&\lesssim |\theta| N^{-\alpha/2} \left[ N^{-1}
 \sum_{x,y =1}^{N-1}
F^N \big( 0, \tfrac xN \big) F^N \big( 0, \tfrac yN \big) {\varphi}^N \big( \tfrac xN , \tfrac yN ; 0,0 \big)\right]^{1/2} \\&+ |\theta| N^{-\alpha/2} \left[ N^{-1} \sum_{x,y =1}^{N-1}
F^N \big( T, \tfrac xN \big) F^N \big( T, \tfrac yN \big) {\varphi}^N \big( \tfrac xN , \tfrac yN ; T,T \big)\right]^{1/2}
\end{split}
\end{equation}
where the space-time density correlation function $\varphi^N$ is defined in \eqref{eq:barphin}.
{Then by assumption on the initial condition, see \eqref{eq:assumption}, and
  by Lemma \ref{lem:appr-num}, we have that 
\begin{equation}
  \label{eq:18}
  \begin{split}
    &N^{-1} \sum_{x,y =1 }^{N-1} F^N \big( 0, \tfrac xN \big) F^N \big( 0, \tfrac yN \big)
    {\varphi}^N \big( \tfrac xN , \tfrac yN ; 0,0 \big) \\
    &\le N^{-2} \sum_{\substack{x,y=1\\x\neq y}}^{N-1} |F^N \big( 0, \tfrac xN \big)| | F^N \big( 0, \tfrac yN \big) |
    + N^{-1} \sum_{x=1 }^{N-1}
    F^N \big( 0, \tfrac xN \big)^2 {\varphi}^N \big( \tfrac xN , \tfrac xN ; 0,0 \big)\\
    &\le \left(N^{-1} \sum_{x=1}^{N-1} \left|F^N \big( 0, \tfrac xN \big)\right| \right)^2
    + C N^{-1} \sum_{x=1}^{N-1} F^N \big( 0, \tfrac xN \big)^2 = o(N^\alpha)
  \end{split}
\end{equation}
}
we get that the first term on the right-hand side of \eqref{eq:tatu} goes to $0$ as $N\to+\infty$. The term involving $F(T,\cdot)$ can be treated analogously  by noting that $F$ is bounded and by using \eqref{eq:boundcorr} with $t=T$.

Now let us focus on the remaining quantity in \eqref{eq:tatu}:
\begin{equation}\begin{split}
\mathbb E\Big[\exp\Big\{-i\theta N^{-\alpha/2} {\mcb M}_T^N (F^N)+\frac{\theta^2}{2}\mathbb V^N (f)\Big\}\exp\Big\{-\frac{\theta^2}{2} \mathbb V^N (f) \Big\}\Big].
\end{split}
\end{equation}
Last display  can be rewritten as
\begin{equation}\label{eq:disp1}
\begin{split}
 &\mathbb E\Big[\exp\Big\{-i\theta N^{-\alpha/2} {\mcb M}_T^N (F^N)+\frac{\theta^2}{2}\mathbb V^N (f)\Big\}\Big(\exp\Big\{-\frac{\theta^2}{2} \mathbb V^N (f)\Big\}- \exp\Big\{-\frac{\theta^2}{2} \mathbb V (f)\Big\}\Big)\Big]\\+& \mathbb E\Big[  \exp\Big\{-i\theta N^{-\alpha/2} {\mcb M}_T^N (F^N)+\frac{\theta^2}{2}\mathbb V^N (f)\Big\} \exp\Big\{\frac{-\theta^2}{2} \mathbb V (f)\Big\}\Big].
\end{split}
\end{equation}
Since on the second  term  of last display we have the exponential complex martingale 
$$ \exp\Big\{-i\theta N^{-\alpha/2} {\mcb M}_T^N (F^N)+\frac{\theta^2}{2}\mathbb V^N (f)\Big\},$$
which has constant expectation equal to one, the second term of \eqref{eq:disp1} is, in fact, equal to $$\exp\Big\{-\tfrac{\theta^2}{2} \mathbb V (f)\Big\}.$$ It remains  to estimate the first term of \eqref{eq:disp1} and show that it goes to zero as $N\to+\infty$. That term can be rewritten as 
\begin{equation}\begin{split}
 \mathbb E\Big[\exp\Big\{-i\theta N^{-\alpha/2} {\mcb M}_T^N (F^N)\Big\}\Big(1- \exp\Big\{\frac{\theta^2}{2}\left[ \mathbb V^N (f) -\mathbb V (f) \right]\Big\}\Big)\Big].
\end{split}
\end{equation}
From the inequality $|1-e^x|\leq |x|e^{|x|}$, plus the fact that $\mathbb V^N (f) \lesssim 1$, $\mathbb V (f) \lesssim 1$ and \eqref{eq:0071},  the proof ends.
\end{proof}

\begin{Lemma}
\label{lem:appr-num}
Let $f\in \mathcal C^\infty(\Omega_T)$ be real-valued. 
 Let $F(t,u)$ be the unique solution of the Laplace equation
\begin{equation}
\label{eq:Dirichlet problem final}
\begin{cases}
&\left(\partial_u^2 F\right) (t,u) =f (t,u), \quad u \in (0,1),\ 0<t<T,\\
&F(t,0)=F(t,1)=0.
\end{cases}
\end{equation}
For any $N\ge 1$, let $F^N(t,\tfrac xN)$,  $x\in \left\{0,  \ldots, N \right\}$,   be the solution of the equation
\begin{equation}
\label{eq:Dirichlet problem2}
\begin{cases}
  &N^{-\alpha} \partial_t F^N(t,\tfrac xN) + \widetilde\Delta_N F^N(t,\tfrac xN) =f(t,\tfrac xN),
   \quad x\in\{1,\ldots, N-1\}, \quad 0<t<T,\\
&F^N(t, 0)= F^N (t, 1)=0,\\
&F^N (T, \tfrac xN) = F(T, \tfrac xN).
\end{cases}
\end{equation}

Then
\begin{equation}
\label{eq:normf}
\cfrac{1}{N} \sum_{z=1}^{N-1} \vert F^N (0, \tfrac zN) \vert^2 = o( N^{\alpha})
\end{equation}
and
\begin{equation}
\label{eq:H_1-app}
\lim_{N\to +\infty} \int_0^T \cfrac{1}{N} \sum_{z=0}^{N-1}
\left|N \left(F^N  (t, \tfrac {z+1} N) - F^N (t, \tfrac z N)\right) -
  \partial_u F \left( t, \tfrac zN \right)  \right|^2 dt =0.
\end{equation}
\end{Lemma}

\begin{proof}
Notice that $F$ can be  explicitly computed and it is given by 
\begin{equation*}
F(t,u)=\int_0^u \left( \int_0^v f(t,w) dw \right) dv -u \left(\int_0^1 \left( \int_0^v f(t,w) dw \right) dv\right).
\end{equation*}
Let $G^N= {F}^N-{F}$ and $\varepsilon^N
=N^{-\alpha} \partial_t F - (\partial_u^2 -\widetilde\Delta_N) F$.
Then $G^N$ satisfies
\begin{equation}
\label{eq:Dirichlet problem6}
\begin{cases}
&N^{-\alpha} \partial_t G^N (t,\tfrac xN)+ \widetilde\Delta_N G^N(t,\tfrac xN) =-\varepsilon^N (t,\tfrac xN), \quad x\in\{1,\ldots, N-1\}, \quad 0< t < T,\\
&G^N(t, 0)= G^N (t, 1)=0,\\
&G^N (T,\cdot) =0.
\end{cases}
\end{equation} 
Multiply the first line of last display  by $G_N$ and
sum over $ x\in\{1,\cdots, N-1\}$.
Since $G^N$ satisfies Dirichlet boundary conditions, we have that 
\begin{equation*}\begin{split}
\sum_{x=1}^{N-1} G^N(t,\tfrac xN) \widetilde\Delta_N G^N(t,\tfrac xN) 
&=-N^2 \sum_{x=0}^{N-1}\Big(G^N(t,\tfrac xN) -G^N(t,\tfrac {x+1}{N})\Big )^2.
\end{split}\end{equation*}
Performing a summation by parts and integrating in time between $0$ and $T$, we get 
\begin{equation}
\label{eq:2897}
\begin{split}
 \cfrac{1}{2N^\alpha} \ \cfrac{1}{N} \sum_{x=1}^{N-1} \left[G^N \left( 0, \tfrac{x}{N} \right) \right]^2 +&\int_0^T \cfrac{1}{N} \sum_{x=0}^{N-1} \left[ N\left( G^N \left( t, \tfrac{x+1}{N} \right) - G^N \left( t, \tfrac{x}{N} \right) \right) \right]^2\ dt \\
&= \int_0^T  \cfrac1N \sum_{x=1}^{N-1}  \varepsilon^N \Big( t, \tfrac xN \Big) G^N \left( t, \tfrac{x}{N} \right)  \ dt .
\end{split} 
\end{equation}
Now let  $E^N:[0,T]\times \left\{0, \tfrac 1N, \ldots, 1 \right\} \to \mathbb R$
be the function defined by 
\begin{equation*}
E^N \left(t, \tfrac xN \right) =\cfrac 1N \sum_{y=0}^x \varepsilon^N \left( t, \tfrac xN\right).
\end{equation*}
Performing a  summation  by parts, we have that 
\begin{equation}\begin{split}
\cfrac1N \sum_{x=1}^{N-1}  \varepsilon^N \Big( t, \tfrac xN \Big) G^N \left( t, \tfrac{x}{N} \right) 
&= - \cfrac1N \sum_{x=0}^{N-1}  E^N \Big( t, \tfrac xN \Big) \left[ N \left( G^N \left( t, \tfrac{x+1}{N} \right)-G^N \left( t, \tfrac{x}{N} \right)\right)\right].\end{split}
\end{equation}
Above we used again the fact that $G^N$ satisfies Dirichlet boundary conditions. 
Since $F$ is smooth, we have for any $x\in\{0,\cdots,N-1\}$, that 
\begin{equation*}
\vert \varepsilon^N \Big( t, \tfrac xN \Big) \vert \lesssim N^{-\alpha}+ N^{-2},
\end{equation*}
and consequently,
\begin{equation*}
\lim_{N \to \infty} \int_0^T  \cfrac1N \sum_{x=0}^{N-1}  \left[ E^N \Big( t, \tfrac xN \Big) \right]^2 dt =0.
\end{equation*}
Hence by Cauchy-Schwarz inequality and \eqref{eq:2897} we conclude that 
\begin{equation*}
\lim_{N\to \infty}  \int_0^T \cfrac{1}{N} \sum_{x=0}^{N-1} \left[ N\left( G^N \left( t, \tfrac{x+1}{N} \right) - G^N \left( t, \tfrac{x}{N} \right) \right) \right]^2dt  =0.
\end{equation*}
Since $F$ is smooth this implies \eqref{eq:H_1-app}.
Using this information in \eqref{eq:2897} 
we get that
\begin{equation*}
  \lim_{N \to \infty}  \cfrac{1}{N^\alpha} \ \cfrac{1}{N}
  \sum_{x=1}^{N-1} \left[G^N \left( 0, \tfrac{x}{N} \right) \right]^2 = 0,
\end{equation*}
and since $F$ is bounded, we have
\begin{equation*}
   \lim_{N \to \infty}  \cfrac{1}{N^\alpha} \ \cfrac{1}{N} \sum_{x=1}^{N-1} \left[F^N \left( 0, \tfrac{x}{N} \right) \right]^2 =0.
\end{equation*}
\end{proof}

\subsection{Tightness}\label{tight}

We recall that $T>0$ is some fixed finite time horizon. The sequence $\{\psi_{n,z} \; ; \; n\in {\mathbb Z}, z\in\mathbb N^*\}$\footnote{$\mathbb N^*:=\mathbb{N}\setminus\{0\}$} of functions on $[0,T]\times [0,1]$, defined by 
\begin{equation}
\label{eq:phinz}
\psi_{n,z}(t,u)=  \sqrt 2\exp\left( 2 i \pi n \tfrac tT\right)\sin\Big( \pi zu\Big), \quad t\in [0,T],\quad  u \in [0,1],
\end{equation}
forms an orthonormal basis of ${L}^2 ([0,T]\times [0,1])$ and for any $f \in {L}^2 ([0,T]\times [0,1])$ we denote $\hat f(n,z) =  \langle f,\psi_{n,z}\rangle$ the Fourier coefficients of $f$ in this basis. Above $\langle\cdot ,\cdot \rangle$ is the inner product in
${L}^2 ([0,T]\times [0,1])$. For $m,k \in \R$, let $\mcb H_{m,k}$ be the Hilbert space obtained as the completion of $C^\infty(\Omega_T, \mathbb C)$ endowed with the inner product{\footnote{Here, $w^*$ denotes
      the complex conjugate of the complex number $w$.}} 
\begin{equation*}
  \langle f ,g \rangle_{m,k}=\sum_{\substack{n\in\mathbb Z\\ z\in \mathbb N^*}}
  (n^2+1)^{m} z^{2k}  \hat f(n,z) \hat g^*(n,z).
\end{equation*}
 The corresponding norm is denoted by $\Vert\cdot \Vert_{m, k}$.
Then for $m,k\ge 0$ we have
$$
\mcb H_{-m,-k} \supset  L^2 ([0,T]\times [0,1]) \supset \mcb H_{m,k}
$$
and $\mcb H_{-m,-k}$ can be identified with the dual  of \,$\mcb H_{m,k}$
with respect to the inner product of $L^2([0,T]\times [0,1])$. The inner product in  $\mcb H_{-m,-k}$ is given by
\begin{equation}\label{eq:norm_dual}
\langle F ,G \rangle_{-m,-k}=\sum_{\substack{n\in\mathbb Z\\ z\in \mathbb N^*}} (n^2+1)^{-m} z^{-2k}  \hat F(n,z) \hat G^*(n,z).
\end{equation}
and we denote  the corresponding norm by 
$\Vert\cdot\Vert_{-m,-k}$.

\begin{Lemma}
\label{lem:tightness}
For  $m, k>1/2$,  the sequence $\{{\mathbb Y}^N\; ; \; N\ge 1\}$ is tight 
  in $\mcb H_{-m,-k}$.
\end{Lemma}

\begin{proof}
We need to show that 
$$
\lim_{A\to+\infty}\limsup_{N\to+\infty}\mathbb P_{\mu_N} (\|{\mathbb Y}^N\|_{-m,-k}>A)=0.
$$
Recall from \eqref{eq:norm_dual} the definition of the $\Vert\cdot\Vert_{-m,-k}$-norm. From Chebychev's inequality the proof ends as long as we show that
\begin{equation}
\sup_{N\ge 1} \mathbb E_{\mu_N} (\|{\mathbb Y}^N\|_{-m,-k}^2) =\sup_{N\ge 1}
\sum_{\substack{n\in\mathbb Z\\ z\in \mathbb N^*}} (n^2+1)^{-m} z^{-2k} \,
\mathbb E_{\mu_N} \left( \vert{\mathbb Y}^N (\psi_{n,z})\vert^2 \right) <\infty.
\label{eq:11}
\end{equation}
Observe that from \eqref{eq:phinz} we have 
\begin{equation*}
 \widetilde \Delta_N \psi_{n,z} (t, \tfrac xN) + N^{-\alpha} \partial_t \psi_{n,z}(t, \tfrac xN) 
  = \left[ 2 i \pi \tfrac nT N^{-\alpha} - 4N^2
    \sin^2 \left(\tfrac {\pi z}{2N}\right)\right]\psi_{n,z}(t, \tfrac xN) . 
\end{equation*}
Applying \eqref{eq:007} with $G= \psi_{n,z}$ we get
\begin{equation}
  \begin{split}
    &{\mathbb Y}^N (\psi_{n,z}) =
    \cfrac{1}{2 i \pi \tfrac nT N^{-\alpha} - 4N^2\sin^2 \left(\tfrac {\pi z}{2N}\right)}
   {\mathbb Y}^N \left( {\widetilde\Delta_N}\psi_{n,z} + N^{-\alpha} \partial_t \psi_{n,z}\right)  \\
    &=   \cfrac{1}{2 i \pi \tfrac nT N^{-\alpha} - 4N^2\sin^2 \left(\tfrac {\pi z}{2N}\right)}
    \left[ \frac{1}{N^{\alpha/2}} \left\{ {\mcb{Y}}^N_T (\psi_{n,z})-
        {\mcb{Y}}^N_0(\psi_{n,z})\right\} -\frac{1}{N^{\alpha/2}} {\mcb M}_T^N(\psi_{n,z})
     \right].
\end{split}
\end{equation}
From \eqref{eq:0071}, we see that 
the quadratic variation of the martingale $\frac{1}{N^{\alpha/2}} {\mcb M}_T^N(\psi_{n,z})$ is given by 
\begin{equation*}
8 N \sin^{2} \left( \tfrac{\pi z}{N} \right)\cos^2\Big(\tfrac{\pi zx}{N}+\tfrac{\pi z}{2N}\Big) \int_0^T \sum_{x=0}^{N-1} \Big(\eta_{tN^{2+\alpha}}(x)  - \eta_{tN^{2+\alpha}}(x+1) \Big)^2 dt \lesssim N^2 \sin^{2} \left( \tfrac{\pi z}{N} \right).
\end{equation*}
Recalling \eqref{eq:boundcorr} and  following similar arguments as in \eqref{eq:18}, we conclude that
\begin{equation}
  \mathbb E \left[ \left\{ {\mcb{Y}}^N_T (\psi_{n,z})-{\mcb{Y}}^N_0(\psi_{n,z})\right\}^2 \right]
  \le 2\mathbb E \left[ |{\mcb{Y}}^N_0(\psi_{n,z})|^2 \right] +2\mathbb E \left[ |{\mcb{Y}}^N_T(\psi_{n,z})|^2 \right] \lesssim 1.
\end{equation}
Collecting the previous bounds together we get that
\begin{equation}
  \label{eq:10}
  \mathbb E \left( \vert{\mathbb Y}^N (\psi_{n,z})\vert^2 \right) \lesssim
   \frac{N^2\sin^2\left(\tfrac{\pi z}{N}\right) + N^{-\alpha}}{n^2 N^{-2\alpha} +N^4 \sin^4 \left(\tfrac{\pi z}{2N}\right)} .
\end{equation}
At this point we need to estimate
\begin{equation*}
\sum_{\substack{n\in\mathbb Z\\ z\in \mathbb N^*}} \frac{1}{(n^2+1)^{m}}\frac{1}{z^{2k}}  \frac{N^2\sin^2\left(\tfrac{\pi z}{N}\right) + N^{-\alpha}}{n^2 N^{-2\alpha} +N^4 \sin^4 \left(\tfrac{\pi z}{2N}\right)} .
\end{equation*}
By symmetry, we split the sum in two cases depending on whether $n=0$ or $n\neq 0$: 
\begin{equation*}
S_1:=\sum_{z\in \mathbb N^*}  \frac{1}{z^{2k}} \frac{N^2\sin^2\left(\tfrac{\pi z}{N}\right) + N^{-\alpha}}{N^4 \sin^4 \left(\tfrac{\pi z}{2N}\right)} .
\end{equation*}
and 
\begin{equation*}
S_2:=2\sum_{{n,z\in \mathbb N^*}} \frac{1}{(n^2+1)^{m}}\frac{1}{z^{2k}}   \frac{N^2\sin^2\left(\tfrac{\pi z}{N}\right) + N^{-\alpha}}{n^2 N^{-2\alpha} +N^4 \sin^4 \left(\tfrac{\pi z}{2N}\right)} .
\end{equation*}
First we observe that for $m>1/2$ we have that
\begin{equation*}
S_2\lesssim \sum_{{n\in \mathbb N^*}} \frac{1}{(n^2+1)^{m}}\sum_{{z\in \mathbb N^*}}\frac{1}{z^{2k}}   \frac{N^2\sin^2\left(\tfrac{\pi z}{N}\right) + N^{-\alpha}}{N^4 \sin^4 \left(\tfrac{\pi z}{2N}\right)}\lesssim S_1,
\end{equation*}
therefore it is enough to estimate $S_1$. 
Note that by writing $z=2Np+q$ and using the periodicity of the function $\sin(\cdot)$ we have 
\begin{equation*}
S_1:=\sum_{p=0}^{+\infty}\sum_{q=1}^{2N-1}\frac{1}{(2Np+q)^{2k}} \frac{N^2\sin^2\left(\tfrac{\pi q}{N}\right) + N^{-\alpha}}{N^4 \sin^4 \left(\tfrac{\pi q}{2N}\right)} .
\end{equation*}
Now we split last sum into 
\begin{equation*}
S^A_{1}:=\sum_{p=0}^{+\infty}\sum_{q=1}^{N}\frac{1}{(2Np+q)^{2k}} \frac{N^2\sin^2\left(\tfrac{\pi q}{N}\right) + N^{-\alpha}}{N^4 \sin^4 \left(\tfrac{\pi q}{2N}\right)} .
\end{equation*}
\begin{equation*}
S_1^B:=\sum_{p=0}^{+\infty}\sum_{q=N+1}^{2N-1}\frac{1}{(2Np+q)^{2k}} \frac{N^2\sin^2\left(\tfrac{\pi q}{N}\right) + N^{-\alpha}}{N^4 \sin^4 \left(\tfrac{\pi q}{2N}\right)} .
\end{equation*}
We start with $S_1^B$. We do the change of variables $r=q-N$ and use a trignometric identity, to get 
\begin{equation*}
S_1^B=\sum_{p=0}^{+\infty}\sum_{r=1}^{N-1}\frac{1}{(2Np+r+N)^{2k}} \frac{N^2\sin^2\left(\tfrac{\pi r}{N}\right) + N^{-\alpha}}{N^4 \cos^4 \left(\tfrac{\pi r}{2N}\right)} .
\end{equation*}
Now we use the fact that for $x\in[0,\tfrac{\pi}{2}]$ it holds that $\cos(x)\geq x-\tfrac \pi2$, to get the bound
\begin{equation*}
S_1^B\lesssim \sum_{p=0}^{+\infty}\sum_{r=1}^{N-1}\frac{1}{(2Np+r+N)^{2k}} \frac{N^2\sin^2\left(\tfrac{\pi r}{N}\right) + N^{-\alpha}}{(r-N)^4} .
\end{equation*}
From the change of variables $v=r-N$ and using again the periodicity of the function $\sin(\cdot)$, we can bound the last display from above by
\begin{equation*}
\begin{split}
S_1^B\lesssim \sum_{p=0}^{+\infty}\sum_{v=-N+1}^{-1}\frac{1}{(2N(p+1)+v)^{2k}} \frac{N^2\sin^2\left(\tfrac{\pi v}{N}\right) + N^{-\alpha}}{v^4}
\end{split}
\end{equation*}
Now, we bound last display from above by  a constant times
$$ \sum_{p=0}^{+\infty}\sum_{v=1}^{N-1}\frac{1}{(2Np+1)^{2k}} \frac{v^2+ N^{-\alpha}}{v^4} \lesssim
1+\sum_{p=1}^{+\infty}\frac{1}{(2Np)^{2k}}<+\infty$$
where we used the fact that $k>1/2$. 
Finally note that
\begin{equation*}
S^A_{1}\lesssim \sum_{p=0}^{+\infty}\sum_{q=1}^{N}\frac{1}{(2Np+q)^{2k}} \frac{q^2 + N^{-\alpha}}{q^4} \lesssim 1+ \sum_{p=1}^{+\infty}\sum_{q=1}^{N}\frac{1}{(2Np)^{2k}} \frac{q^2 + N^{-\alpha}}{q^4}<+\infty.
\end{equation*}
This ends the proof.

\end{proof}

\section{Generalization and discussion}

Theorem \ref{th:mainthm1} can be adapted without difficulties to the case of periodic boundary conditions at equilibrium, i.e. by considering the dynamics evolving on the discrete torus (without reservoirs) $\mathbb T_N=\{0,1,\ldots, N-1\}$ starting from its equilibrium state. Since the mass of the system  is conserved  we have to restrict the space of test functions $f$ that satisfy $\int_\mathbb T f(t,u)du=0$ for any $t\in[0,T]$. Here $~\mathbb T$ denotes the one-dimensional torus.    In this case the sequence $\{\mathbb Y^N\; ; \; N \ge 1\}$ converges to a field $\mathbb Y$ white in time and correlated in space according to the kernel
{$\rho(1-\rho)(-\partial_u^2)^{-1}$, where $\rho$ is the limit average density.}

In fact, in the case of periodic conditions, at equilibrium, Theorem  \ref{th:mainthm1} can be proved similarly for various interacting particle systems: the Kipnis-Marchioro-Presutti model \cite{KMP}, the symmetric Ginzburg-Landau dynamics with harmonic potential \cite{GPV,Y},  the Symmetric Inclusion Process \cite{GRV},  independent random walks  (i.e. the Symmetric Zero-Range Process with linear rate) and Generalized Exclusion Process \cite{FGS}.

\bigskip

In the sequel we see briefly how to extend our main theorem in the case of generic symmetric Zero-Range processes at equilibrium. Our result is however limited to a time scale $t N^{2+\alpha}$ with $\alpha<4$. More precisely we consider a Zero-Range process  $\{\eta_{tN^{2+\alpha}} \; ; \; t\ge 0\}$ with state space $\Omega^N={\mathbb N}^{\mathbb T_N}$ with generator $\mcb L $ acting on a test function $f:\Omega^N \to \mathbb R$ as
\begin{equation*}
(\mcb L f )(\eta) =\sum_{\substack{x,y\in \mathbb T_N\\ |x-y|=1}}\  g (\eta(x)) \left[ f(\eta^{x,y}) -f(\eta) \right].
\end{equation*}
Here $g:\mathbb N \to \mathbb N$ has to satisfy some technical assumptions listed in \cite{GJS} in order to assure the validity of the second order Boltzmann-Gibbs bound \eqref{eq:sobg}. We refer the reader to \cite{GJS} for details. We denote by $\nu_\rho$ the equilibrium probability measure with density $\rho$. Let us consider the dynamics at equilibrium under $\nu_\rho$.

We define
\begin{equation*}
  \mathbb Y^N (f) = N^{\alpha/2} \int_0^T {\mcb Y}_t^N (f) dt
\end{equation*}
where the random fluctuation field $\{ {\mcb {Y}}_t^N\; ; \;  t\ge 0\}$
acts on a space-time function $f: [0,T]\times \mathbb T \to \mathbb R$ satisfying 
$\int_\mathbb T f(t,u)du=0$ for any $t\in[0,T]$, as
\begin{equation*}
  {\mcb { Y}}_t^N (f) := \cfrac{1}{\sqrt N} \sum_{x\in \mathbb T_N}
  {f}( t, \tfrac xN) (\eta_{t N^{2+\alpha}} (x) -\rho).
\end{equation*}
Let
\begin{equation*}
V_g (\eta (x)) = g(\eta (x)) - {\tilde g} (\rho) - {\tilde g}^{\prime} (\rho) (\eta (x) -\rho). 
\end{equation*}
By Dynkin's formula it follows that the process $\{  {\mcb M}_t^N( F) \; ; \; t \ge 0\}$ defined by
\begin{equation}
\begin{split}
  {\mcb M}_t^N( F)&=\mcb{Y}^N_t( {{F}})-{\mcb {Y}}^N_0( {{ F}}) \\
  &-\int_0^t \cfrac{1}{\sqrt N} \sum_{x \in \mathbb T_N}
  {\partial_s F}( s, \tfrac xN) (\eta_{s N^{2+\alpha}} (x) - \rho) ds\\
  &- N^{\alpha -1/2}  \int_0^t \sum_{x\in \mathbb T_N} {\tilde g}^{\prime} (\rho) (\widetilde\Delta_N F)( s, \tfrac xN)
  \left[ \eta_{sN^{2+\alpha}} (x) -\rho \right] ds \\
  &- N^{\alpha -1/2}  \int_0^t \sum_{x\in {\mathbb T_N}} (\widetilde\Delta_N F)( s, \tfrac xN)
  V_g \left(\eta_{sN^{2+\alpha}} (x)  \right) ds 
  \end{split}
\end{equation}
is a martingale w.r.t. the natural filtration of $\{\eta_{tN^{2+\alpha}} \; ; \; t\ge 0\}$. 

Let $\mathbb T=[0,1)$ be the unit torus. Consider now $f: (t,u) \in [0,T]\times \mathbb T \to f(t,u) \in \mathbb R$ be a  function satisfying $\int_\mathbb T f(t,u)du=0$ for any $t\in[0,T]$ and consider, for any $N\ge 1$, the function $F^N: (t,u) \in [0,T]\times \tfrac{1}{N} {\mathbb T}_N \to F^N(t,u) \in \mathbb R$  satisfying  the equation
\begin{equation}
\label{eq:Dirichlet problemZR0}
\begin{split}
&N^{-\alpha} \partial_t F^N(t,\tfrac xN) + {\tilde g}^\prime (\rho) \widetilde\Delta_N F^N(t,\tfrac xN) =f(t,\tfrac xN),\quad x\in\{1,\ldots, N-1\},\\
&F^N (T,\cdot) =0.
\end{split}
\end{equation}
We have then 
\begin{equation}
\label{eq:007ZR}
\begin{split}
{\mathbb{ Y}}^N \left(f \right)
&=- \frac{1}{N^{\alpha/2}}{\mcb{Y}}^N_0({ F^N}) -\frac{1}{N^{\alpha/2}} {\mathcal M}_T^N(F^N)\\
&+ N^{(\alpha-1)/2}  \int_0^T \sum_{x\in {\mathbb T_N}} (\widetilde\Delta_N F^N)( s, \tfrac xN) V_g \left(\eta_{sN^{2+\alpha}} (x)  \right) ds 
\end{split}
\end{equation}
By adapting  Remark 11 in \cite{GJ0} we have that for any $\ell \ge 1$
\begin{equation}
\label{eq:sobg}\begin{split}
&\mathbb E \left[ \left(  \int_0^T \sum_{x\in {\mathbb T_N}} (\widetilde\Delta_N F^N)( s, \tfrac xN)
    V_g \left(\eta_{sN^{2+\alpha}} (x)  \right) ds \right)^2\right] \\&\lesssim
\left[ \frac{\ell}{N^{2+\alpha}} + \frac{T}{\ell}  \right] \
\int_0^T \sum_{x\in {\mathbb T_N}} \left[(\widetilde\Delta_N F^N) \left(s, \tfrac xN\right) \right]^2 ds.\end{split}
\end{equation}
By taking $\ell =N^{1+\alpha/2}$ and using \eqref{eq:Dirichlet problemZR0} we get 
\begin{equation}
\begin{split}
  &\mathbb E \left[ \left(  \int_0^T \sum_{x\in {\mathbb T_N}} (\widetilde\Delta_N F^N)( s, \tfrac xN)
      V_g \left(\eta_{sN^{2+\alpha}} (x)  \right) ds \right)^2\right] \\
  & \lesssim N^{-(1+\alpha/2)}  \ \int_0^T \sum_{x\in {\mathbb T_N}} \left[(\widetilde\Delta_N F^N) \left(s, \tfrac xN\right)
  \right]^2 ds\\
  &\lesssim N^{-(1+\alpha/2)}  \left\{  \int_0^T \sum_{x\in {\mathbb T_N}} \left[f \left(s, \tfrac xN\right) \right]^2 ds
    + N^{-2\alpha}  \int_0^T \sum_{x\in {\mathbb T_N}} \left[\partial_s F^N \left(s, \tfrac xN\right) \right]^2 ds \right\}.
\end{split}
\end{equation}

Hence the square of the $L^2$-norm of the third term on the right-hand side of \eqref{eq:007ZR} is bounded above by $N^{\alpha/2-2}$. It goes to $0$ if, and only if, $\alpha<4$. This ends the proof.
 \\

\backmatter

\bmhead{Acknowledgments}

P.G. thanks  FCT/Portugal for financial support through the projects UIDB/04459/2020 and UIDP/04459/2020.  This project has received funding from the European Research Council (ERC) under  the European Union's Horizon 2020 research and innovative programme (grant agreement   n. 715734).

\section*{Declarations}


\begin{itemize}
\item Funding: UIDB/04459/2020, UIDP/04459/2020, ERC n. 715734.
\item Conflict of interest/Competing interests: The authors declare that there are no conflict of interest.
\item Availability of data and materials: Not applicable.
\item Code availability: Not applicable 
\end{itemize}
%

\bibliographystyle{amsalpha}

\end{document}